\documentclass[b5paper] {article}
[10pt]

\usepackage{algorithm}    

\usepackage{graphicx}
\graphicspath{ {images/} }    
    
\makeatletter
\renewcommand*\l@section{\@dottedtocline{1}{1.5em}{2.3em}}
\makeatother

\usepackage{amsfonts}
\usepackage{amssymb}
\usepackage[T1]{fontenc}

\usepackage{tikz}
\usetikzlibrary{calc}

\usepackage{CJK}
\usepackage{amsmath}
 
\usepackage{amsfonts}
\usepackage{amssymb}
\usepackage{amsthm}
\usepackage{amssymb}
\usepackage{enumerate}
\usepackage[calc]{picture}
\usepackage[all,cmtip]{xy}

\usepackage[mathscr]{eucal}
\usepackage{eqlist}

\usepackage{color}
\usepackage{abstract} 
\usepackage[T1]{fontenc}
 
\theoremstyle{plain}
\newtheorem{theorem}{Theorem}

\newtheorem{example}[theorem]{Example}
\newtheorem{corollary}[theorem]{Corollary}

\theoremstyle{definition}

\usepackage{etoolbox}
\newtheoremstyle{myrem}
 {3pt}
 {3pt}
 {\normalsize}
 { }
 {\itshape}
 {:}
 { }
 {}

 \theoremstyle{myrem}
 \newtheorem{remark}{Remark}
 \appto\remark{\leftskip\parindent}
 \appto\remark{\rightskip\parindent}

\numberwithin{equation}{section}
\numberwithin{theorem}{section}

\begin{document}

\begin{center}
{\Large\bf Isometric embeddings of graphs into Riemannian manifolds}

\bigskip

\footnotetext[1]
{{\bf{AMS Mathematical Classifications 2010 (2000)}}. 	05C12,  		53C22	 }


Shiquan Ren

\medskip
 
 $\text{\ }$

\smallskip

\parbox{24cc}{{\small

{\textbf{Abstract}.}  
An isometric embedding of a graph  into a metric space is an embedding of the vertices   such that the smallest number of edges connecting any two vertices equals to the  distance in the metric space between the images. In this paper, we study isometric embeddings of graphs into Riemannian manifolds. We give some classifications of  graphs that can be isometrically embedded into certain Riemannian manifolds.  
}}

\end{center}

\section{Introduction}

In data science, graph is an important model   to characterize  relations among points.  For example,  networks, point-clouds in metric spaces, etc. can be investigated by graphical models.  In some  problems of data science,  the points of  data are taken from a Riemannian manifold, for example, neural network  \cite{ml,ml-1}, image and visualization \cite{com3,com1}, etc.  The underlying relations among the points in the data could be represented by a graph. 

\smallskip

A {graph} $G=(V(G),E(G))$ is a pair  of sets such that  the elements of $E(G)$ are $2$-element subsets of $V(G)$ (cf. \cite[p. 2]{gr}). Elements of $V(G)$ are called {vertices}. And elements of $E(G)$ are called {edges}. A graph $G$ is called {finite} if both $V(G)$ and $E(G)$ are finite sets. The {degree} of a vertex $v$, denoted as $\deg(v)$,  is the number of edges that contain $v$ (cf. \cite[p. 5]{gr}).  A graph $G$ is called {\bf complete} if each pair of distinct vertices is an edge (cf. \cite[p. 3]{gr}).  Given two graphs $G_1$ and $G_2$, we say that $G_1$ is a {subgraph} of $G_2$ if $V(G_1)\subseteq V(G_2)$ and $E(G_1)\subseteq E(G_2)$ (cf. \cite[pp. 3-4]{gr}).

Let $G$ be a graph.    For any two distinct vertices $x,y\in V(G)$, a {path} $\gamma$ from $x$ to $y$ is a sequence of edges
$\sigma_1\sigma_2\ldots\sigma_k$,    $\sigma_i=\{v_{i-1},v_{i}\}$ for  $i=1,2,\ldots,k$,  such that
$x=v_0$, $y=v_{k}$, 
and 
 $v_0,v_1,\ldots,v_{k}$ are distinct (cf. \cite[p. 6]{gr}). 
 We call $k$ the {length} of $\gamma$, and denote $k$ as $|\gamma|$. 
For any $x,y\in V(G)$, the {\bf distance}  
$d_G(x,y)$ is defined  to  be the length of the  shortest path  in $G$ from $x$ to $y$  (cf. \cite[p. 8]{gr}). 
 If there does not exist any path from $x$ to $y$, then we define $d_G(x,y)=\infty$. We define $d_G(x,x)=0$ for all $x\in V(G)$. It follows from a straight-forward verification that $(V(G),d_G)$ is a metric space.  An {\bf isometric embedding} of $G$ into a metric space $(X,d_X)$ is an injective map $f: V(G)\longrightarrow X$ such that for any $x,y\in V(G)$, 
\begin{eqnarray}\label{eq888}
d_G(x,y)=d_X(f(x),f(y)). 
\end{eqnarray}

The isometric embeddings of graphs into metric spaces have been studied in \cite{iso3,iso4,iso5,iso1,iso2}.  In 1970's, the isometric embeddings of graphs into Cartesian products of complete graphs were studied by Graham and Pollak \cite{iso4,iso5} to find addressing schemes for communications networks. In 1984, the topic was further studied by Winkler \cite{iso2}.  In 1985, isometric embeddings of graphs into Cartesian products of certain discrete metric spaces were studied by Graham and Winkler \cite{iso1}.  And in 1988, an investigation of isometric embeddings of graphs into metric spaces was given by Graham \cite{iso3}.


 \smallskip

Let $M$ be a complete Riemannian manifold (cf. \cite{do}). Let $d_M$ be the distance function on $M$ such that for any points $p,q\in M$, $d_M(p,q)$ is the length of the shortest geodesic from $p$ to $q$. If there does not exist any geodesic on $M$ from $p$ to $q$, then we define $d_M(p,q)=\infty$. We define $d_M(p,p)=0$ for all $p\in M$.

  For any distinct points $p,q\in M$, if there exist at least two distinct shortest geodesics $\gamma$ and $\gamma'$ of length $l$, both of which are from $p$ to $q$, then we say that $p$ and $q$ are {\bf dual points} of distance $l$.

\smallskip

In this paper,  we study isometric embeddings of graphs into Riemannian manifolds. In Theorem~\ref{pr5.1},  we prove that if a connected finite graph has a vertex of degree greater than or equal to $3$,   and the graph can be isometrically embedded into a complete Riemannian manifold without dual points of distance $2$,  then the graph must be  complete.  As a consequence, we give a classification of connected finite graphs that can be isometrically embedded into complete Riemannian manifolds without dual points of distance $2$, in Corollary~\ref{co5.2}.  In Theorem~\ref{pr5}, we give a classification of graphs isometrically embedded into spheres, with a vertex of degree greater than or equal to $3$. We also prove that the isometric embedding is unique up to an isometric homeomorphism of the sphere.  As a consequence, we give a classification of graphs isometrically embedded into spheres, in Corollary~\ref{co88}.  

The remaining part of this paper is organized as follows. In Section~\ref{s2},  we prove Theorem~\ref{pr5.1} and Theorem~\ref{pr5}. 
In Section~\ref{s4}, we propose some questions.


\section{Classifications of isometric embeddings of graphs into Riemannian manifolds}\label{s2}

In this section, we give some classification results of graphs that can be isometrically embedded into Riemannian manifolds. We prove Theorem~\ref{pr5.1}, derive Corollary~\ref{co5.2}, and give Example~\ref{ex-3} in Subsection~\ref{s2.1}.  We prove Theorem~\ref{pr5} and derive Corollary~\ref{co88} in Subsection~\ref{s2.2}.  

\subsection{Isometric embeddings of graphs into Riemannian manifolds without dual points of distance $2$}\label{s2.1}

In the following theorem, we investigate the graphs that can be isometrically embedded into a Riemannian manifold without dual points of distance $2$.

\begin{theorem}\label{pr5.1}
Let $G$ be a connected finite graph such that there exists a vertex $v$ with $\deg (v)\geq 3$.  Let $M$ be a  complete Riemannian manifold without dual points of distance $2$.  If there exists an isometric embedding of $G$ into $M$, then $G$ is a complete graph. 
\end{theorem}

\begin{proof}
Let $f: V(G)\longrightarrow M$ be an isometric embedding of $G$ into $M$.  Since $\deg (v)\geq 3$, we let $a$, $b$ and $c$ be distinct vertices of $G$ such that $\{v,a\}$, $\{v,b\}$ and $\{v,c\}$ are edges of $G$. Since $f$ is isometric, there exist shortest geodesics $\gamma_{f(v),f(a)}$ from $f(v)$ to $f(a)$,  $\gamma_{f(v),f(b)}$ from $f(v)$ to $f(b)$,  and $\gamma_{f(v),f(c)}$ from $f(v)$ to $f(c)$ in $M$ such that the lengths of $\gamma_{f(v),f(a)}$, $\gamma_{f(v),f(b)}$, and $\gamma_{f(v),f(c)}$ are all $1$.  Consequently, there exists a point  among $f(a)$, $f(b)$ and $f(c)$, say $f(a)$, such that $f(v)$ is a cusp of the product of $\gamma_{f(v),f(a)}$ and $\gamma_{f(v),f(b)}$, and also a cusp of the product of $\gamma_{f(v),f(a)}$ and $\gamma_{f(v),f(c)}$.  Consequently, 
\begin{eqnarray}\label{eq-1}
d_M(f(a),f(b))<|\gamma_{f(v),f(a)}|+|\gamma_{f(v),f(b)}|=2,\\
d_M(f(a),f(c))<|\gamma_{f(v),f(a)}|+|\gamma_{f(v),f(c)}|=2. \label{eq-2}
\end{eqnarray}
Since $d_G(a,b)$ and $d_G(a,c)$ are integers and $a, b, c$ are distinct, with the help of (\ref{eq888}), it follows from (\ref{eq-1}) and (\ref{eq-2}) respectively that
\begin{eqnarray}\label{eq2}
&d_M(f(a),f(b))=d_G(a,b)=1,\\
&d_M(f(a),f(c))=d_G(a,c)=1. \label{eq2.1}
\end{eqnarray}
Therefore, there exist shortest geodesics $\gamma_{f(a),f(b)}$ from $f(a)$ to $f(b)$, and $\gamma_{f(a),f(c)}$ from $f(a)$ to $f(c)$, such that the lengths of both $\gamma_{f(a),f(b)}$ and $\gamma_{f(a),f(c)}$ are $1$. 

In the next, we show that each pair of distinct vertices among $v,a,b,c$ is an edge of $G$. By (\ref{eq2}) and (\ref{eq2.1}) respectively, both $\{a,b\}$ and $\{a,c\}$ are edges of $G$. Hence we only need to show that $\{b,c\}$ is  an edge of $G$.  We divide the proof in two cases. 

{\sc Case~1}. $f(v)$ is a cusp of the product of $\gamma_{f(v),f(b)}$ and $\gamma_{f(v),f(c)}$. 

 Then 
\begin{eqnarray*}
d_M(f(b),f(c))< |\gamma_{f(v),f(b)}|+|\gamma_{f(v),f(c)}|=2. 
\end{eqnarray*}
Thus with the help of (\ref{eq888}), 
\begin{eqnarray}\label{eq3}
d_M(f(b),f(c))=d_G(b,c)=1. 
\end{eqnarray}
It follows from (\ref{eq3}) that $\{b,c\}$ is  an edge of $G$.

{\sc Case~2}. $f(v)$ is not a cusp of the product of $\gamma_{f(v),f(b)}$ and $\gamma_{f(v),f(c)}$. 

Then there exists a geodesic $\gamma_{f(b),f(v),f(c)}$  of length $2$, starting at $f(b)$, ending at $f(c)$, and passing through $f(v)$, such that  $f(v)$ is the middle point  of  $\gamma_{f(b),f(v),f(c)}$ in length.  Hence
\begin{eqnarray}\label{eq9}
d_M(f(b),f(c))\leq |\gamma_{f(b),f(v),f(c)}|=2. 
\end{eqnarray}
Suppose to the contrary, $\{b,c\}$ is not   an edge of $G$. Then 
\begin{eqnarray}\label{eq10}
d_M(f(b),f(c))=d_G(b,c)\geq 2.
\end{eqnarray} 
Hence by (\ref{eq9}) and (\ref{eq10}), we have
\begin{eqnarray}\label{eq11}
d_M(f(b),f(c))= 2.  
\end{eqnarray}
Thus $\gamma_{f(b),f(v),f(c)}$ is a  shortest geodesic on $M$.

On the other hand, since both $\gamma_{f(b),f(a)}$ and $\gamma_{f(a),f(c)}$ are geodesics of length $1$, we see that the product of $\gamma_{f(b),f(a)}$ and $\gamma_{f(a),f(c)}$ gives a path on $M$ from $f(b)$ to $f(c)$, of length $2$. Consequently, it follows with the help of (\ref{eq11}) that the product of $\gamma_{f(b),f(a)}$ and $\gamma_{f(a),f(c)}$ is   a shortest geodesic of length $2$.       This contradicts the assumption that $M$ has no dual point. Therefore, $\{b,c\}$ is  an edge of $G$.

Summarizing both {\sc Case~1} and {\sc Case~2}, we see that $\{b,c\}$ is  an edge of $G$. Hence each pair of distinct vertices among $v,a,b,c$ is an edge of $G$.

Let $d$ be a vertex such that at least one of $\{d,a\}$, $\{d,b\}$, $\{d,c\}$, and $\{d,v\}$ is an edge of $G$. Without loss of generality, we suppose $\{d,a\}$ is an edge of $G$.  By applying the above arguments to $\{d,a,b,c\}$, $\{d,a,b,v\}$, and $\{d,a,c,v\}$ subsequently, we have that each pair of distinct vertices among $\{d,a,b,c\}$, $\{d,a,b,v\}$, and $\{d,a,c,v\}$ is an edge of $G$. That is, each pair of distinct vertices among $\{d,a,b,c,v\}$ is an edge of $G$.

Since $G$ is connected and finite, by an induction on the set of vertices of $G$,  we have that any two distinct vertices of $G$ is an edge of $G$. Hence $G$ is the complete graph. 
\end{proof}

The next corollary gives a complete classification of connected finite graphs that can be isometrically embedded into a Riemannian manifold $M$ without dual points of distance $2$.  It is a consequence of Theorem~\ref{pr5.1}. 

\begin{corollary}\label{co5.2}
Suppose there is an isometric embedding of a connected finite graph $G$ into a complete Riemannian manifold $M$ without dual points of distance $2$. Then one of the following holds:
\begin{enumerate}[(a).]
\item
$G$ is $\{\{v_1,v_2\},\{v_2,v_3\},\ldots,\{v_{n-1},v_n\}\}$ for some $n\geq 1$;
\item
$G$ is $\{\{v_1,v_2\},\{v_2,v_3\},\ldots,\{v_{n-1},v_n\},\{v_1,v_n\}\}$ for some $n\geq 3$;
\item
$G$ is the complete graph. 
\end{enumerate}
\end{corollary}

\begin{proof}
We divide the proof in two cases.

{\sc Case~1}. Every vertex $v$ of $G$ satisfies $\deg (v)\leq 2$.

\noindent Then since $G$ is connected and finite, $G$ is either (a) or (b).

{\sc Case~2}. There exists a vertex $v$ of $G$ with $\deg (v)\geq 3$.

\noindent Then by Theorem~\ref{pr5.1}, $G$ is (c). 
\end{proof}

\begin{remark}
By the Hadamard Theorem (cf. \cite[Theorem 4.5]{hd} and \cite[Chap. 7]{do}),  if $M$ is a simply-connected complete Riemannian manifold with non-positive sectional curvature,  then any two points in $M$ are connected by a unique shortest geodesic. Hence $M$ has no dual points. And the classifications in Theorem~\ref{pr5.1} and Corollary~\ref{co5.2} hold. 
\end{remark}

Without the assumption that $M$ has no dual points of distance $2$, the classifications of $G$ in Theorem~\ref{pr5.1} and Corollary~\ref{co5.2} may not hold.  We give such an example in the following. 




\begin{example}\label{ex-3}
Let $n\geq 2$. Let $S^n$ be the $n$-sphere in $\mathbb{R}^{n+1}$. Suppose $S^n$ has  center $0$ and radius $\frac{2}{\pi}$. We consider a subset of $S^n$
\begin{eqnarray*}
V=\{(\pm\frac{2}{\pi},0,\ldots,0),(0,\pm\frac{2}{\pi},0,\ldots,0),\ldots,(0,\ldots,0,\pm\frac{2}{\pi})\}. 
\end{eqnarray*}
 Let $G_n$ be the graph with its set of vertices $V$. For any $a,b\in V$, we let $\{a,b\}$ be an edge of $G$ if and only if $\langle a,b\rangle=0$, or equivalently, $a,b$ are not antipodal. Here $\langle \text{\ } , \text{\ }\rangle$ is the canonical inner product of $\mathbb{R}^{n+1}$. 
Then we have an isometric embedding of $G_n$ into $S^n$.

In particular, when $n=2$, $G_2$ is the $1$-skeleton of the regular octahedron with vertices $(0,0,\pm\frac{2}{\pi})$, $(0,\pm\frac{2}{\pi},0)$, and $(\pm\frac{2}{\pi},0,0)$. 
\end{example}

\subsection{Isometric embeddings of graphs into spheres}\label{s2.2}

The next theorem gives the classification of the graphs that can be isometrically embedded into spheres, as well as the classification of the isometric embeddings. 

\begin{theorem}\label{pr5}
Let $G$ be a graph such that there exists a vertex $v$ with $\deg(v)\geq 3$. Suppose $G$ can be isometrically embedded into the $n$-sphere $S^n$ with radius $r$,  for some $n\geq 2$. Then 
\begin{enumerate}[(a).]
\item
either $G$ is a complete graph with the number of vertices smaller than or equal to $n+2$, and $r=\frac{1}{\text{arccos}(-\frac{1}{3})}$,   or  $G$ is a subgraph of  $G_n$ where $G_n$ is the graph given in Example~\ref{ex-3}, and $r=\frac{2}{\pi}$;
\item
for any two isometric embeddings $f$ and $g$ of $G$ into $S^n$, there exists an isometric homeomorphism $\varphi$ from $S^n$ to itself such that $f=\varphi\circ g$. 
\end{enumerate} 
\end{theorem}

\begin{proof}

(a).  Let $f$ be an isometric embedding of $G$ into $S^n$. Then $G$ is connected and finite. We divide the proof of (a) into two steps.  

{\bf Step~1}. 
 Suppose $a$, $b$ and $c$ are distinct vertices of $G$ such that $\{a,v\}$, $\{b,v\}$ and $\{c,v\}$ are edges of $G$.  Then we have  the geodesics $\gamma_{f(v),f(a)}$, $\gamma_{f(v),f(b)}$ and $\gamma_{f(v),f(c)}$ of lengths $1$. By applying an analogous argument in Theorem~\ref{pr5.1},  the distances of at least two pairs among $f(a)$, $f(b)$ and $f(c)$ are strictly less than $2$.  It follows that at least two of $d_{S^n}(f(a),f(b))$, $d_{S^n}(f(a),f(c))$ and $d_{S^n}(f(b),f(c))$ are $1$. 
Without loss of generality, we suppose $d_{S^n}(f(a),f(c))$ and $d_{S^n}(f(b),f(c))$ are $1$. 

{\sc Case~1}. $d_{S^n}(f(a),f(b))=1$. 

 Then each pair of the vertices $v,a,b,c$ is an edge of $G$.  Moreover, $f(v)$, $f(a)$, $f(b)$, $f(c)$ are the vertices of a regular tetrahedron in a totally geodesic $2$-sphere of $S^n$.  Consequently,  the radius of $S^n$ is 
\begin{eqnarray}\label{eq-8}
r=\frac{1}{\text{arccos}(-\frac{1}{3})}. 
\end{eqnarray}

{\sc Case~2}. $d_{S^n}(f(a),f(b))=2$. 

 Then $\gamma_{f(a),f(v),f(b)}$ and $\gamma_{f(a),f(c),f(b)}$ are both geodesics from $f(a)$ to $f(b)$ of length $2$.  Hence $f(a)$ and $f(b)$ are antipodal. 
 Consequently, the radius of $S^n$ is
\begin{eqnarray}\label{eq-9}
r=\frac{2}{\pi}. 
\end{eqnarray}
Moreover, The mid-point of $\gamma_{f(a),f(v),f(b)}$ is $f(v)$ and the mid-point of $\gamma_{f(a),f(c),f(b)}$ is $f(c)$.

{\bf Step~2}. (i). Suppose {\sc Case~1} holds for $v,a,b,c$.

Then we have (\ref{eq-8}). Let $x$ be a vertex of $G$ such that $\{x,a\}$ is an edge of $G$. 
Then by applying the argument of {\sc Case~1} and {\sc Case~2} in Step~1 to $x,a,b,c$, we have that $\{x,a\}$, $\{x,b\}$ and $\{x,c\}$ are all edges of $G$.  And by applying the argument of {\sc Case~1} and {\sc Case~2} in Step~1 to $x,a,b,v$, we have that $\{x,v\}$ is also an edge of $G$.  Hence each pair of $x,a,b,c,v$ is an edge of $G$.  By an induction on the number of vertices of $G$, we have that $G$ is a complete graph.  We notice that $G$ is the $1$-skeleton of a simplex, and the dimension of this simplex is smaller than or equal to $n+1$.  This implies that the number of vertices of $G$ is smaller than or equal to $n+2$.

(ii). Suppose {\sc Case~2} holds for $v,a,b,c$.

Then we have (\ref{eq-9}). Let $x$ be a vertex of $G$ such that $\{x,a\}$ is an edge of $G$. Then by applying the argument of {\sc Case~1} and {\sc Case~2} to $x,a,c,v$, it follows that {\sc Case~2} holds for $x,a,c,v$. Hence $\gamma_{f(a),f(x)}$ is in the same big circle of either $\gamma_{f(c),f(a)}$ or $\gamma_{f(v),f(a)}$. Equivalently, $f(x)$ is  the antipodal point of either $f(c)$ or  $f(v)$. By an induction on the number of vertices of $G$, the image of the isometric embedding of $V(G)$ is a subset of $V$ given in Example~\ref{ex-3}.  Hence $G$ is a subgraph of $G_n$.

Summarizing both (i) and (ii) in Step~2, we obtain (a). 

\smallskip

(b). Let $f$ and $g$ be two isometric embeddings of $G$ into $S^n$. We divide the proof into two cases. 

{\sc Case~1}. $G$ is a complete graph. 

Let $\Delta^{n+1}$ be the regular $(n+1)$-simplex in $\mathbb{R}^{n+1}$ with its vertices in $S^n$.  Let $V(\Delta^{n+1})$ be the set of vertices of $\Delta^{n+1}$. 
Let $\text{sk}^1(\Delta^{n+1})$ be  the $1$-skeleton of $\Delta^{n+1}$.  We regard $\text{sk}^1(\Delta^{n+1})$ as a graph. Then $G$ is a complete subgraph of $\text{sk}^1(\Delta^{n+1})$. That is, $G$ is the $1$-skeleton of a face of $\Delta^{n+1}$.  Any isometric embedding $h$ of $G$ into $S^n$ can be extended  to be an isometric embedding $\tilde h$ of $\text{sk}^1(\Delta^{n+1})$ into $S^n$.  Let $\tilde f$ and $\tilde g$ be the extended isometric embeddings of $f$ and $g$ respectively.  Then there exists an 
isometric homeomorphism $\varphi$ from $S^n$ to itself such that 
$\tilde f=\varphi\circ \tilde g$. 
 By restricting $\tilde f$ and $\tilde g$ to $G$ respectively, we have $f=\varphi\circ g$. 

{\sc Case~2}. $G$ is a subgraph of $G_n$. 

Then any isometric embedding $h$ of $G$ into $S^n$ can be extended  to be an isometric embedding $\tilde h$ of $G_n$ into $S^n$.   Moreover,  for any two isometric embeddings $\tilde f$ and $\tilde g$ of $G_n$ into $S^n$, there  exists an isometric homeomorphism $\varphi$ from $S^n$ to itself such that $\tilde f=\varphi\circ \tilde g$.  
By an analogous argument with {\sc  Case~1}, we have that there exists an isometric homeomorphism $\varphi$ from $S^n$ to itself such that  $f=\varphi\circ g$.

Summarizing both {\sc Case~1} and {\sc Case~2}, we obtain (b). 
\end{proof}

The next corollary gives a complete classification of  all the graphs that can be isometrically embedded into spheres. It is a direct consequence of  Theorem~\ref{pr5}~(a). The proof is analogous to the proof of Corollary~\ref{co5.2}. 

\begin{corollary}\label{co88}
Suppose there is an isometric embedding of a   graph $G$ into   $S^n$, for some $n\geq 2$. Then one of the following holds:
\begin{enumerate}[(a).]
\item
$G$ is $\{\{v_1,v_2\},\{v_2,v_3\},\ldots,\{v_{n-1},v_n\}\}$ for some $n\geq 1$;
\item
$G$ is $\{\{v_1,v_2\},\{v_2,v_3\},\ldots,\{v_{n-1},v_n\},\{v_1,v_n\}\}$ for some $n\geq 3$;
\item
$G$ is a complete graph with the number of vertices smaller than or equal to $n+2$;
\item
$G$ is a subgraph of $G_n$. 
\end{enumerate}
\end{corollary}

\section{Further discussions}\label{s4}

As further discussions, we propose the following questions.
\begin{enumerate}[$\textit{Question}$ 1.]
\item
 Given an arbitrary Riemannian manifold $M$, how to classify all the graphs $G$ that can be isometrically embedded into $M$? 
\item
  Given an arbitrary graph $G$, does there exist a Riemannian manifold $M$ such that $G$ can be isometrically embedded into $M$?  If such $M$ exists, how to construct $M$?
\item
 Given an arbitrary Riemannian manifold $M$ and a graph $G$ that can be isometrically embedded into $M$, how to classify all the isometric embeddings of $G$   into $M$? 
\end{enumerate}

\bigskip

\noindent {\bf Acknowledgement}. The author would like to express his deep gratitude to  Prof. Jie Wu, Prof. Stephane Bressan, and Prof. Tee-How Loo  for their  guidance and encouragement. 

\bigskip

Address: School of Computing, National University of Singapore.

E-mail: sren@u.nus.edu












\end{document}